\documentclass{amsart}

\usepackage{amssymb,amsmath,amsthm,amscd}
\usepackage{babel}
   \usepackage [utf8]{inputenc}
   \usepackage [T1] {fontenc}

\usepackage{pdfpages}

\newtheorem{theorem}{Theorem}[section]
\newtheorem{theorem-definition}[theorem]{Theorem-Definition}

\newtheorem{proposition}{Proposition}[section]

\newtheorem{lemma}{Lemma}[section]
\newtheorem{remark}{Remark}[section]

\newcommand{\ZZ}{\mathbb{Z}}

\newcommand{\RR}{\mathbb{R}}
\newcommand{\CC}{\mathbb{C}}

\numberwithin{equation}{section}
\begin{document}

\title[On the positive powers of $q$-analogs of Euler series]{On the positive powers of $q$-analogs of Euler series}
\author{Changgui ZHANG}
\address{ Laboratoire P. Painlev\'e (UMR -- CNRS 8524), D\'epartement Math., FST, Universit\'e de Lille, Cit\'e scientifique, 59655
Villeneuve d'Ascq cedex, France}
\email{changgui.zhang@univ-lille.fr}
\thanks{This work was partially supported by Labex CEMPI (Centre Européen pour les
Mathé\-ma\-tiques, la Physique et leurs Interaction).}

\begin{abstract}
The most simple and famous divergent power series coming from ODE may be the so-called Euler series $\sum_{n\ge 0}(-1)^n\,n!\,x^{n+1}$, that, as well as all its positive powers, is Borel-summable in any direction excepted the negative real half-axis (see \cite{Ba} or \cite{Ma}).  
By considering a family of linear $q$-difference operators associated with a given first order non-homogenous  $q$-difference equation, it will be shown that  the summability order of $q$-analoguous counterparties of Euler series depends upon of the degree of power under consideration.
\end{abstract}
\maketitle
 \section{Introduction}
 
It is well-known that, given any direction starting from the origin of the complex plan, the set of all Borel-summable series in this direction forms a differential algebra; see \cite[Proposition 1.3.4.2]{Ma} or \cite[Theorem 36]{Ba}. This is not the same for the summability with respect to the power series solutions of $q$-difference equations. Indeed, if $q>1$ and \begin{equation}
 \label{equation:Eqx}
\hat E_q(x)=\sum_{n\ge 0}(-1)^nq^{n(n-1)/2}x^n\,,
\end{equation}
the divergent power series $\hat E_q(x)$ is $Gq$-summable of only one level but its square $\hat E_q(x)^2$ is $Gq$-summable of two levels; see \cite[\S4.3.8]{Zhang1} and \cite[Th\'eor\`eme 2.2.1]{MZ}.  The first purpose of this paper is to obtain a linear $q$-difference equation for every positive power of the power series $\hat E_q(x)$. This will allow one to see the different level of summability for each of these corresponding power series. 

In what follows, let $q$ be a given nonzero complex number, and let $\sigma_q$ be the associated $q$-difference operator defined by $\sigma_qf(x)=f(qx)$. By direct computation, one finds the series defined by \eqref{equation:Eqx} in the above satisfies term by term the following linear $q$-difference equation :
\begin{equation}
\label{equation:Euler}
 (x\sigma_q+1)y=1\,.
\end{equation}
Now, write \eqref{equation:Euler} into the form $x\sigma_qy=1-y$ and then square both sides. As $(\sigma_qy)^2=\sigma_q(y^2)$, one obtains that
$
x^2\sigma_q(y^2)-y^2=1-2y
$. By using again \eqref{equation:Euler}, one gets the following equation:
$$
(x\sigma_q+1)(x^2\sigma_q-1)(y^2)=x+1-2(x\sigma_q+1)y\,,
$$
what implies that
 \begin{equation}
  \label{equation:E2} (x\sigma_q+1)(x^2\sigma_q-1)\hat E_q(x)^2=x-1\,.
 \end{equation}
 
 We shall explain how to obtain a linear $q$-difference equation for $\hat E_q(x)^n$. For doing that, we shall start by recalling some non-commutative rings of $q$-difference operators and then arrive at a $n$-th order such operator associated with any given $n$-th power of a power series solution of some first order non-homogenous $q$-difference  equation. See Theorem \ref{thm:fn} in \S~\ref{sec:thm} and its generalisation stated in Theorem \ref{thm:gen} in \S~\ref{sec:gen}. Section \ref{sec:proof} is reserved to the proofs of Theorem \ref{thm:fn} and two related Lemmas. In the last section, we shall apply, respectively, both Theorems \ref{thm:fn} and  \ref{thm:gen} to the above-mentioned power series $\hat E_q(x)$ and an other $q$-analog $\hat E(x;q)$ of the classic Euler series. At the end of the paper will be outlined some results about the summability of both $\hat E_q(x)$ and  $\hat E(x;q)$ when $q>1$; see Theorem \ref{thm:summability}. 
  
 Throughout the present paper, we shall limit ourself to the $q$-summation method studied in our previous works \cite{Zhang1} and \cite{MZ}, by means of a $q$-Laplace integral involving the multivalued function $e^{\log^2x/2\ln q}$. There exists another summation, using a Jacobi theta function; see \cite{RSZ}. Comparing both summations can be seen as part of Stokes analysis, and this would be interesting to find identities on the associated special functions.

 \section{Notation and statements}\label{sec:thm}
 As usual, we will denote by $\CC\{x\}$ or $\CC[[x]]$ the integral domain constituting of the germs of holomorphic functions at $x=0$ in the complex plan or that of the power series of indeterminate $x$, respectively. Their respective corresponding field of fractions will be denoted by $\CC[x^{-1}]\{x\}$ or $\CC[x^{-1}][[x]]$. One has $\CC\{x\}\subset\CC[[x]]$, as well as  $\CC[x^{-1}]\{x\}\subset\CC[x^{-1}][[x]]$, in a natural way via the classic Taylor or Laurent series.
 
 Let $R$ be one of the above-mentioned $\CC$-vector spaces $\CC\{x\}$, $\CC[[x]]$, $\CC[x^{-1}]\{x\}$ and $\CC[x^{-1}][[x]]$. By replacing $x$ with $qx$,
 the $q$-difference operator $\sigma_q$ acts as being an automorphism over $R$. Let  $\sigma_q^{k}=\sigma_{q^k}$ for any integer $k$; in particular, $\sigma_q^0$ is simply the identity map on $R$. We will denote by $R[\sigma_q]$ the set of the (linear) $q$-difference operators whose coefficients belong to $R$, {\it i.e.} $L\in R[\sigma_q]$ if $L=\sum_{k=0}^na_k\sigma_q^k$, where $a_k\in R$. By identifying each element of $R$ with the corresponding multiplication map in $R$, the set $R[\sigma_q]$ constitutes itself a non-commutative ring for the compoisition of operators. Given $L_1$, $L_2\in R[\sigma_q]$, we will write $L_1L_2$ in stead of $L_1\circ L_2$.
 
 Furthermore, we will denote by $\nu_0(f)\in\ZZ\cup\{\infty\}$ the valuation of any given element $f$ of the field $ \CC[x^{-1}][[x]]$ ``at $x=0$''. This means that (i) $\nu_0(f)=\infty$ iff $f$ is the identically vanishing series, and (ii) $v_0(f)=\nu\in\ZZ$ iff $f\equiv c\,x^{\nu}$ $\bmod$ $x^{\nu+1}\,\CC[[x]]$ for some nonzero complex number $c$.
 
 Throughout the whole section, $P$ will be some given nonzero power series belon\-ging to the field $\CC[x^{-1}][[x]]$. Define the associated family of power series $\{P_n\}_{n\ge 0}$ as follows: $P_0=0$, $P_1=P$, and
  \begin{equation}
   \label{equation:Pn}P_n=\sum_{k=0}^{n-1}(-x\sigma_q)^kP\quad\textrm{for}\quad n\ge 2.
  \end{equation}
  By noticing that $(-x\sigma_q)^kP=(-x)^k\,q^{k(k-1)/2}\,\sigma_q^kP$, it follows from 
   \eqref{equation:Pn} that, for any pair $(n,\ell)\in\ZZ_{\ge 0}\times\ZZ_{>0}$, 
  \begin{align*}
  P_{n+\ell}-P_{n}&=\sum_{k=n}^{n+\ell-1}(-x)^k\,q^{k(k-1)/2}\,\sigma_q^kP\\
  &\equiv (-1)^n\,c\,q^{n\nu+n(n-1)/2}x^{\nu+n}\ \bmod\ x^{\nu+n+1}\,\CC[[x]]\,,
\end{align*}
where $c\not=0$ and $\nu=\nu_0(P)\in\ZZ$.
 This implies  that
  \begin{equation}
   \label{equation:Pmn}P_m\not= P_{n}\qquad \textrm{if}\qquad m\not=n.
  \end{equation}

Given any positive integer $n\ge 1$, let $\{L_{n,k}^P\}_{1\le k\le n}\subset\left(\CC[x^{-1}][[x]]\right)[\sigma_q]$ be the associated family of $q$-difference operators defined in the following manner:
  \begin{align}
   \label{equation:Ln1}L_{n,1}^P&=\frac{1}{P_1}\left(x^n\sigma_q-(-1)^n\right)=\frac{1}{P}\left(x^n\sigma_q-(-1)^n\right);\\  \label{equation:Lnk}
    L_{n,k+1}^P&=\frac{1}{P_{k+1}}\left(x^{n-k}\sigma_q-(-1)^{n-k}\right)\,L_{n,k}^P,\quad  1\le k<n\,.
  \end{align}
 Specifically, letting $k=n$, one gets the following form of the $n$-th order $q$-difference operator $L_{n,n}^P\in\left(\CC[x^{-1}][[x]]\right)[\sigma_q]$:
  \begin{equation}
   \label{equation:Ln}
   L_{n,n}^P=\frac{1}{P_n}\left(x\sigma_q+1\right)\frac{1}{P_{n-1}}\left(x^2\sigma_q-1\right)\cdots\frac{1}{P_1}\left(x^n\sigma_q-(-1)^n\right)\,.
  \end{equation}

\begin{theorem}\label{thm:fn}
 Let $n$ be an integer $\ge 1$,  let $P$, $f\in\CC[x^{-1}][[x]]$ with $P\not=0$, and suppose that the following $q$-difference equation is satisfied: 
 \begin{equation}
  \label{equation:f}
  x\sigma_qf+f=P\,.
 \end{equation}
 Then:
 \begin{equation}
  \label{equation:fn}
  L_{n,n}^P(f^n)=(-1)^{n(n-1)/2}\,,
 \end{equation}
 where $L_{n,n}^P$ is the $q$-difference operator defined by \eqref{equation:Ln}. 
\end{theorem}

Theorem \ref{thm:fn}
will be proved with the help of the following lemmas.

\begin{lemma}\label{lem:Akj}
 Let $(n,k)\in \ZZ^2$ such that $1\le k\le n$, and let $P$, $f$ be as in Theorem \ref{thm:fn}. For ${0\le j\le k}$, define $A_{k;j}\in\CC[x^{-1}][[x]]$ by
 \begin{equation}
  \label{equation:Akj}
  A_{k;j}=\prod_{\substack{\ell\not=j\\ 0\le\ell\le k}}\left( P_j-P_\ell\right).
 \end{equation}
One has:
\begin{equation}
 \label{equation:Lnkf}
 L_{n,k}^P(f^n)=(-1)^{k(2n-k+1)/2}\sum_{j=0}^k\frac{1}{A_{k;j}}\left(f-P_j\right)^n\,.
\end{equation} 
\end{lemma}

\begin{lemma}
 \label{lem:sym} Let $n$ be a nonnegative integer, $\mathfrak{C}$ an extension field of $\CC$, and let $(\alpha_0,\alpha_1,...,\alpha_n)\in\mathfrak{C}^{n+1}$ such that $\alpha_j\not=\alpha_{\ell}$ for $j\not=\ell$. Define the $(n+1)$-uplet $ (a_0,a_1,...,a_n)\in\mathfrak{C}^{n+1}$ as follows: 
 $$
 a_{j}=\prod_{\substack{\ell\not=j\\ 0\le\ell\le n}}\left( \alpha_j-\alpha_\ell\right)
 $$
 ($a_0=1$ if $n=0$).
 The following identity holds in the ring $\mathfrak{C}[T]$ of the polynomial functions over $\mathfrak{C}$:
 \begin{equation}
  \label{equation:Faj}
  \sum_{j=0}^n\frac{1}{a_j}\,(T-\alpha_j)^n=(-1)^{n}\,.
 \end{equation}

\end{lemma}

\section{Proofs}\label{sec:proof}
\begin{proof}[\bf Proof of Lemma \ref{lem:Akj}] To simplify the presentation, we will write $L_{n,k}$ in stead of $L_{n,k}^P$, and set  
$\displaystyle\epsilon_{n,k}=(-1)^{k(2n-k+1)/2}$.
One has:
\begin{equation}
 \label{equation:ekj}
 \epsilon_{n;k+1}=(-1)^{n-k}\,\epsilon_{n,k}\,.
\end{equation}

Let us proceed by induction on $k$.
For $k=1$, since $P_1=P$, one has: 
$$\epsilon_{n,1}=(-1)^n;\quad 
A_{1;0}=-P\,,\quad A_{1;1}=P\,.$$
Thus, \eqref{equation:Ln1} implies that
$$L_{n,1}(f^n)=\frac{1}{P}\left(x^n\sigma_q(f^n)-(-1)^nf^n\right)\,.
$$
At the same time, by considering \eqref{equation:f}, one gets that 
$$
x^n\sigma_q (f^n)=(x\sigma_q f)^n=(P-f)^n\,;
$$
In this way, it follows that
$$
L_{n,1}(f^n)=\frac{\epsilon_{n,1}}{P}\left((f-P)^n-f^n\right)\,,
$$
that is exactly what wanted in \eqref{equation:Lnkf}
for $k=1$.

Now, suppose that equality \eqref{equation:Lnkf} holds for some integer $k$ between $1$ and $n-1$. By using relation \eqref{equation:Lnk}, one can express $L_{n,k+1}(f^n)$ as follows:
\begin{equation}
 \label{equation:Lk+1}
L_{n,k+1}(f^n)=\frac{\epsilon_{n,k}}{P_{k+1}}\sum_{j=0}^k\left(B_{n,k;j}-C_{n,k;j}\right)\,,
\end{equation}
where 
\begin{equation}
 \label{equation:B} B_{n,k;j}=x^{n-k}\sigma_q\left(\frac{1}{A_{k;j}}\left(f-P_j\right)^n\right)
\end{equation}
and
\begin{equation}
 \label{equation:C}  C_{n,k;j}=\frac{(-1)^{n-k}}{A_{k;j}}\left(f-P_j\right)^n.
\end{equation}
In view of \eqref{equation:Pn}, one finds that $x\sigma_qP_j=P-P_{j+1}$, hence:
$$x\sigma_q(f-P_j)=x\sigma_q f-x\sigma_qP_j=(P-f)-(P-P_{j+1})=-(f-P_{j+1}).
$$
Thus, it follows that 
$$
B_{n,k;j}=\frac{\left(x\sigma_q(f-P_j)\right)^n}{x^k\sigma_qA_{k;j}}=\frac{(-1)^n}{x^k\sigma_qA_{k;j}}(f-P_{j+1})^n\,.
$$
By using the expression of $A_{k;j}$ given in \eqref{equation:Akj}, one has :
$$
x^k\sigma_qA_{k;j}=\prod_{\substack{\ell\not=j\\ 0\le\ell\le k}}\left(x\sigma_q( P_j-P_\ell)\right)=(-1)^k\prod_{\substack{\ell\not=j\\ 0\le\ell\le k}}\left( P_{j+1}-P_{\ell+1}\right).
$$
This is to say that
\begin{equation*}
 \label{equation:Akj+1}
 x^k\sigma_qA_{k;j}=\frac{(-1)^k}{P_{j+1}}\,A_{k+1,j+1}\,.
\end{equation*}
So, one can write \eqref{equation:B} as follows:
\begin{equation}
 \label{equation:Bkj+1}
 B_{n,k;j}=\frac{(-1)^{n-k}P_{j+1}}{A_{k+1;j+1}}(f-P_{j+1})^n\,.
\end{equation}

Rewrite \eqref{equation:Lk+1} into the following form:
\begin{equation}
 \label{equation:Lk+1a}
L_{n,k+1}(f^n)=\frac{\epsilon_{n,k}}{P_{k+1}}\left(\sum_{j=0}^{k-1}\left(B_{n,k;j}-C_{n,k;j+1}\right)+B_{n,k;k}-C_{n,k;0}\right)\,.
\end{equation}
Let $j<k$; one has: 
$$A_{k+1,j+1}=(P_{j+1}-P_{k+1})\,A_{k,j+1}.
$$
Replace $j$ by $j+1$  in \eqref{equation:C}, and make use of \eqref{equation:Bkj+1}. One gets that
$$
B_{n,k;j}-C_{n,k;j+1}=\frac{(-1)^{n-k}P_{k+1}}{A_{k+1;j+1}}(f-P_{j+1})^n\,.
$$
Furthermore, one sees that $A_{k+1;0}=-P_{k+1}\,A_{k;0}$.
By letting $j=k$ and $j=0$ in \eqref{equation:Bkj+1} and \eqref{equation:C} respectively, one has:
$$
B_{n,k;k}=\frac{(-1)^{n-k}P_{k+1}}{A_{k+1;k+1}}(f-P_{k+1})^n$$
and
$$
C_{n,k;0}=-\frac{(-1)^{n-k}P_{k+1}}{A_{k+1;0}}(f-P_0)^n\,.
$$
Thus, in view of  \eqref{equation:ekj}, equality \eqref{equation:Lk+1a} implies that
$$
L_{n,k+1}(f^n)=\epsilon_{n,k+1}\,\sum_{j=0}^{k+1}\frac{1}{A_{k+1;j}}\left(f-P_j\right)^n\,,
$$
which corresponds to
\eqref{equation:Lnkf} in which $k$ was replaced by  $k+1$.
\end{proof}

\begin{proof}[\bf Proof of Lemma \ref{lem:sym}] We give here two proofs.

{\it Proof by induction --}
If $n=0$, as $a_0=1$,  equality \eqref{equation:Faj} becomes evident.
Suppose that \eqref{equation:Faj} holds for some index $n\ge 0$, and consider an $(n+2)$-uplet $(\alpha_0,...,\alpha_{n+1})$ where, as before, $\alpha_j\not=\alpha_\ell$ if $j\not=\ell$. Set
$$
\mathcal{P}(T)=\sum_{j=0}^{n+1}\frac{1}{\tilde a_j}\,(T-\alpha_j)^{n+1},
\quad  \tilde a_j=\prod_{\substack{\ell\not =j\\ 0\le \ell\le n+1}}(\alpha_j-\alpha_\ell)
.$$
Let $k$ be an integer such that $0\le k\le n+1$. If one writes $$\displaystyle 
a_{k;j}=\prod_{\substack{\ell\not =k,\ell\not =j\\ 0\le \ell\le n+1}}(\alpha_j-\alpha_\ell)
\,,$$ one has $a_{k;j}=\tilde a_j/(\alpha_j-\alpha_k)$ for $j\not=k$, what implies that
\begin{equation}
 \label{equation:Pk}
\mathcal{P}(\alpha_k)=-\sum_{\substack{j\not =k\\ 0\le j\le n+1}}\frac{1}{a_{k;j}}\,(\alpha_k-\alpha_j)^{n}.
\end{equation}
By applying the induction hypothesis to the $(n+1)$-uplet $(\alpha_0,...,\alpha_{k-1},\alpha_{k+1},...\alpha_{n+1})$ appeared in the right hand side of \eqref{equation:Pk}, one obtains that  $\mathcal{P}(\alpha_k)=(-1)^{n+1}$.
As $\mathcal{P}$ is a polynomial with $\deg \mathcal{P}\le n+1$, one finds that $\mathcal{P}(T)=(-1)^{n+1}$ identically. 

\medskip

{\it Proof by Lagrange polynomials\footnote{We would like to thank our friend and collaborator J. Sauloy for this elegant proof.} --} Let $\mathbb{K}=\mathfrak{C}(T)$ be the field of the rational functions over $\mathfrak{C}$. Let $X$ be a new indeterminate, and set
$$F(X) = (T - X)^n,\quad  \Lambda_j(X) = \prod_{\substack{\ell\not =j\\ 0\le \ell\le n}}(X- \alpha_\ell)\,.
$$
By noticing that $F(X)\in\mathbb{K}[X]$ be such that $\deg F = n < n+1$, the Lagrange interpolation formula \cite[Chap. IV, \S 2, p. 26]{Bourbaki} implies that
$$
F(X) = \sum_{j=0}^n\frac{ F(\alpha_j)}{a_j} \Lambda_j(X)\,.
$$
And now one compares the  coefficients of $X^n$ for both sides of the above equality. As that of $\Lambda_j$ is equal to 1, one finds that
$$
(-1)^n =\sum_{j=0}^n\frac{F( \alpha_j)}{a_j}= \sum_{j=0}^n\frac{(T - \alpha_j)^n}{a_j},
$$
which is exactly the expected identity \eqref{equation:Faj}.
\end{proof}

\begin{proof}[\bf End of the proof of Theorem \ref{thm:fn}]
In view of Lemma \ref{lem:Akj}, letting $k=n$ in \eqref{equation:Lnkf} yields that
 $$
 L_{n,n}^P\left(f^n\right)=(-1)^{n(n+1)/2}\sum_{j=0}^n\frac{1}{A_{n;j}}\left(f-P_j\right)^n\,,
 $$
 where
 $$
  A_{n;j}=\prod_{\substack{\ell\not=j\\ 0\le\ell\le n}}\left( P_j-P_\ell\right).
 $$
 By
 \eqref{equation:Pmn}, one knows  that  $P_j\not=P_\ell$ for  $j\not=\ell$. Thus, by applying Lemma \ref{lem:sym} with $\mathfrak{C}=\CC[x^{-1}][[x]]$, $X=f$ and $\alpha_j=P_j$, one gets that
 $$
 \sum_{j=0}^n\frac{1}{A_{n;j}}\left(f-P_j\right)^n=(-1)^{n}.
 $$
 This permits to finish the proof of Theorem \ref{thm:fn}.
\end{proof}

\section{One generalisation of Theorem \ref{thm:fn}}\label{sec:gen}

Instead of  \eqref{equation:f}, let us consider the following slightly more general $q$-difference equation:
\begin{equation}
  \label{equation:ab}
  \alpha\,\sigma_qy+y=\beta\,,
 \end{equation}
 where both $\alpha$ and $\beta$
 are nonzero power series belonging to the field $\CC[x^{-1}][[x]]$. In the same spirit as in \eqref{equation:Pn}, define the sequence $\{\beta_n\}_{n\ge 0}$ in $\CC[x^{-1}][[x]]$ as follows:
 \begin{equation}
  \label{equation:beta}
  \beta_0=0;\qquad \beta_n=\sum_{k=0}^{n-1}\left(-\alpha\,\sigma_q\right)^k\beta\,,\quad n\ge 1\,.
 \end{equation}

 \begin{proposition}
  \label{prop:ab} For any given positive integer $n$, the following conditions are equivalent.
  \begin{enumerate}
  \item $\beta_j\not= 0$ for any positive integer $j$ such that $0< j\le n$. 
   \item $\beta_j\not=\beta_\ell$ for any couple  of integers $(j,\ell)$ such that $0\le \ell<j\le n$. 
  \end{enumerate}

 \end{proposition}

 \begin{proof}

 Indeed, by \eqref{equation:beta}, it follows that 
 \begin{equation}
  \label{equation:betajk}
 \beta_j-\beta_\ell=\sum_{k=\ell}^{j-1}\left(-\alpha\,\sigma_q\right)^k\beta=(-\alpha\,\sigma_q)^{\ell}\beta_{j-\ell}\,.
 \end{equation}
 As $\alpha\not=0$, l'operator $(-\alpha\,\sigma_q)$ is an automorphism on the $\CC$-vector space $\CC[x^{-1}][[x]]$. So, this is the same for its $\ell$-th power or iteration $(-\alpha\,\sigma_q)^{\ell}$. This implies the equivalence between the conditions stated in Proposition \ref{prop:ab}.  
 \end{proof}

 Given a nonzero $f\in\CC[x^{-1}][[x]]$, one remembers that $\nu_0(f)$ denotes the valuation of $f$ at $x=0$, that is the lowest degree of the terms  of $f$.
 By using \eqref{equation:beta}, one obtains that, for $n>0$:
 $$
 \nu_0(\beta_n)=\left\{ 
 \begin{array}{ll}
  \nu_0(\beta)&\textrm{if}\ \nu_0(\alpha)>0;\\
  (n-1)\nu_0(\alpha)+\nu_0(\beta)&\textrm{if}\ \nu_0(\alpha)<0.
 \end{array}
\right.
 $$
 This gives the following statement for $\alpha$, $\beta\in\CC[x^{-1}][[x]]\setminus\{0\}$.
 
 \begin{remark}
  \label{rem:ab} Both conditions considered in Proposition \ref{prop:ab} are necessarily fulfilled if $\nu_0(\alpha)\not=0$. 
 \end{remark}

 \begin{theorem}
  \label{thm:gen}  Let $n\in\ZZ_{>0}$, and let $\alpha$, $\beta$ and $f\in\CC[x^{-1}][[x]]$. Suppose that $f$ satisfies the $q$-difference equation in \eqref{equation:ab} and that $\beta_j\not= 0$ for any positive integer $j$ such that $0< j\le n$. One has:
\begin{equation}
 \label{equation:abfn}
 L_{n}^{\alpha,\beta}(f^n)=(-1)^{n(n-1)/2}\,,
\end{equation}
where $L_{n}^{\alpha,\beta}$ is the $n$-th order $q$-difference operator defined by
  \begin{equation}
   \label{equation:abL}
   L_{n}^{\alpha,\beta}=\frac{1}{\beta_n}\left(\alpha\,\sigma_q+1\right)\frac{1}{\beta_{n-1}}\left(\alpha^2\,\sigma_q-1\right)\cdots \frac{1}{\beta_1}\left(\alpha^n\,\sigma_q-(-1)^n\right)\,.
  \end{equation}
 \end{theorem}
 
 \begin{proof} Replace $(k,P)$ with $(n,\beta)$ in  \eqref{equation:Akj}, and define:
  $$A_{n;j}=\prod_{\substack{\ell\not=j\\ 0\le\ell\le n}}\left( 
  \beta_j-\beta_\ell\right)\,,\quad 0\le j\le n.
  $$
  By taking into account Proposition \ref{prop:ab}, it follows that $A_{n;j}\not=0$. Thus, one might proceed in the same way as for the proof of Theorem   \ref{thm:fn}. We omit the details.
 \end{proof}

\section{About the summability of the powers of $q$-Euler series}\label{sec:Rq}

Let us come back to the power series $\hat E_q(x)$ defined by \eqref{equation:Eqx}, which satisfies the $q$-difference equation stated in \eqref{equation:Euler}. Letting $P=1$, Theorem \ref{thm:fn} implies immediately the following result.
\begin{remark}
Given any integer $n\ge 2$, the $n$-th power $\hat E_q(x)^n$ satisifies the following identity:
\begin{equation}
 \label{equation:Eulern}
\left(x\sigma_q+1\right)\,
 \frac1{P_{n-1}}\,\left(x^2\sigma_q-1\right)\,...(x^n\sigma_q-(-1)^n)\hat E_q(x)^n=(-1)^{n(n-1)/2}\,P_n\,,
\end{equation}
where $P_k= 1-x+...+q^{(k-1)(k-2)}(-x)^{k-1}$ for $1\le k\le n$.

In particular, when $n=2$, as $P_2=1-x$, \eqref{equation:Eulern} takes the form of   \eqref{equation:E2}. 
\end{remark}

Furthermore, it might be amusing to notice that, if $q=1$, one has $\displaystyle\hat E_1(x)=\frac{1}{1+x}$ and, in this case, the above identity in \eqref{equation:Eulern} is simply equivalent to the following elementary relation:
$$
\prod_{k=1}^n\frac{x^k-(-1)^k}{\sum_{j=0}^{k-1}(-x)^j}=(-1)^{n(n-1)/2}\,(x+1)^n\,.
$$ This observation related with the specific case of $q=1$ may be also made for Theorem \ref{thm:gen}, by assuming, for example, both $\alpha$ and $\beta$ to belong to the field $\CC[x^{-1}]\{x\}$ of the germs of meromorphic functions at $x=0$ in $\CC$ or to the sub-field $\CC(x)$ of the rational functions over $\CC$.

Our next remark goes to another $q$-analog of the following so-called {Euler series}:
\begin{equation}
 \label{equation:serieEuler}
\hat E(x) = \sum_{n\ge 0}(-1)^n\,n!\,x^{n+1}\,.
\end{equation}
This power series is divergent for all $x\in\CC\setminus\{0\}$ but Borel-summable in every direction excepted $\RR^-$. It satisfies the following first order ODE:
\begin{equation}
  \label{equation:eqEuler}
(x\,\delta+1) y=x, \quad \delta=x\partial_x=x\frac{d\ }{dx}.
 \end{equation}
 Letting $Y=\hat E(x)^2$, one can check that
\begin{equation}
 \label{equation:eqEuler2}
 (x\,\delta+1-x)\,(x\,\delta+2)Y=2x^2.
\end{equation}
Indeed, if $y=\hat E(x)$, one deduces from \eqref{equation:eqEuler} that $x\delta(y^2)=2x\,y\,\delta y=2y(x-y)=2xy-2y^2$. This means that $(x\delta+2)(y^2)=2x\,y$. Thus, applying again \eqref{equation:eqEuler} yields that
$\displaystyle
(x\,\delta+1)\,\frac{1}{2x}\,(x\,\delta+2)\,y^2=x\,.
$ In this way, one gets \eqref{equation:eqEuler2}, using the identity $\displaystyle(x\,\delta+1)\,\frac{1}{2x}=\frac{1}{2x}\,(x\,\delta+1-x)$ in the non-commutative ring  $\CC(x)[\delta]$ of the differential operators over $\CC(x)$. 

Besides, the correspondances
 $\displaystyle
 n \leftrightarrow \frac{1-q^n}{1-q}$ and $\displaystyle n! \leftrightarrow \frac{(1-q)...(1-q^n)}{(1-q)^n}
 $ suggest one to consider
 the following $q$-analog of the Euler series:
 \begin{equation}
  \label{equation:Exq}
\hat E(x;q)=x+
 \sum_{n\ge 1}(-1)^n\frac{(1-q)...(1-q^n)}{(1-q)^n}x^{n+1}.
 \end{equation}
 This can be written in term of a basic hypergeometric  series as follows: 
 $$\hat E(x;q)=\displaystyle x\,{}_2\phi_1(q,q;0;q,-\frac{x}{1-q})=x\,{}_2\phi_0(q^{-1},q^{-1};-;q^{-1},\frac{x}{1-q^{-1}}).
 $$
 See \cite[(1.2.22), p. 4]{GR} for the general definition of ${}_r\phi_s(...;...;q,z)$.
 
 Let  $\displaystyle \delta_q=\frac{\sigma_q-1}{q-1}=x\,\Delta_q$, where $\Delta_q$ is defined in \cite[p. 488, (10.2.3)]{AAR}. By observing that
 $\displaystyle n=\frac{\delta x^n}{x^n}$ and $\displaystyle \frac{1-q^n}{1-q}=\frac{\delta_q x^n}{x^n}$, a direct computation shows that
 $\hat E(x;q)$ satisfies the $q$-analog of \eqref{equation:eqEuler} as follows:
 $\displaystyle
 (x\,\delta_q+1)y=x\,.
 $ If one writes
 \begin{equation}
  \label{equation:ab1}\alpha=\frac{x}{q-1-x}\,,\quad \beta=(q-1)\,\alpha\,,
 \end{equation}it follows from the above that
 \begin{equation}
  \label{equation:Eulerq}\left(\alpha\,\sigma_q+1\right)\hat E(x;q)=\beta\,.
 \end{equation} 

 \begin{remark}
  \label{rem:qEuler2} Let $\alpha$ be as in \eqref{equation:ab1}. The  following identity holds in the field $\CC[x^{-1}][[x]]$:
  \begin{equation}
   \label{equation:qEuler2a}\left(\alpha\,\delta_q+\frac{1}{q-1-x}\right)\,\left(\alpha\,\delta_q-\frac1x+\frac{1}{q-1-x}\right)\hat E(x;q)^2=\alpha\,(\sigma_q\alpha-1)\,.
  \end{equation}
Moreover, when $q\to1$, \eqref{equation:qEuler2a} is reduced into the following equivalent form of  \eqref{equation:eqEuler2}:
\begin{equation}
 \label{equation:eqEuler2a}
 \left(\delta+\frac{1}{x}\right)\,\left(\delta+\frac{2}{x}\right)\hat E(x)^2=2\,.
\end{equation}

 \end{remark}
 
 To obtain \eqref{equation:qEuler2a}, one can apply \eqref{equation:beta} for $n=1$ and $n=2$, where $\alpha$ and $\beta$ are defined by \eqref{equation:ab1}. This gives that
 $$\beta_1=(q-1)\,\alpha,\quad \beta_2=\beta-\alpha\,\sigma_q\beta=(q-1)\,\alpha\,\left(1-\sigma_q\alpha\right)\,.
 $$
 Thus, it follows from applying Theorem \ref{thm:gen} to \eqref{equation:Eulerq} that $\hat E(x;q)^2$ satisfies the following $q$-difference equation:
\begin{equation}
  \label{equation:Euler2La}
 L\left(\hat E(x;q)^2\right)=\alpha\,(\sigma_q\alpha-1)\,,
 \end{equation}
 where 
 \begin{equation}
  \label{equation:Euler2L}
 L=\frac{1}{(q-1)^2}\,\left(\alpha\,\sigma_q+1\right)\,\left(\alpha\,\sigma_q-\frac{1}{\alpha}\right)\,.
 \end{equation}
 By replacing $\sigma_q$ with $(q-1)\delta_q+1$ in \eqref{equation:Euler2L}, one gets that
 \begin{equation*}
  \label{equation:Euler2Lb}
 L=\left(\alpha\,\delta_q+\frac{1+\alpha}{q-1}\right)\,\left(\alpha\,\delta_q+\frac{\alpha^2-1}{(q-1)\,\alpha}\right).
 \end{equation*}
 As $\displaystyle\alpha+1=\frac{q-1}{q-1-x}$ and $\displaystyle\alpha^2-1=\frac{(q-1)(2x-q+1)}{(q-1-x)^2}$, one deduces immediately \eqref{equation:qEuler2a} from \eqref{equation:Euler2La}. The limit equation form given in \eqref{equation:eqEuler2a} can be obtained from \eqref{equation:qEuler2a}  by noticing both limits $\alpha\to-1$ and $\delta_q\to\delta$ for $q\to1$.

In the rest of this paper, we will suppose that $q>1$ even if, in most cases, the hypothesis $|q|>1$ may be really enough; see \cite{Ro} and \cite{RSZ}. Let us recall some results about the summability of power series solutions of $q$-difference equations. Given any linear $q$-difference operator $L$ of the following form:
$$
L=\sum_{j=0}^na_j\sigma_q^j\in\CC\{x\}[\sigma_q]\,,\quad a_0a_n\not=0\,,
$$
one definies its associated Newton polygon $\mathcal{NP}(L)$ as being the convex hull of the set $\{(j,m):m\ge \nu_0(a_j)\}$ in the strip $[0,n]\times [0,+\infty)$. 
 By following \cite[(3.1.2)]{MZ} and \cite[Proposition 5.14]{Zhang1}, one knows that, 
if   $\mathcal{NP}(L)$ admits only integer slopes, saying $\kappa_1\le \kappa_2\le ...\le \kappa_n$, then there exist  $\nu\in\ZZ$, $\left(h_1, h_2, ..., h_n\right)\in\left(1+x\CC\{x\}\right)^n$ and $(c_0,c_1, ..., c_n)\in\left(\CC\setminus\{0\}\right)^{n+1}$ such that
\begin{equation}
 L=c_0\,x^\nu\,h_1\,\left(x^{\kappa_1}\sigma_q+c_1\right)\,h_2\,\left(x^{\kappa_2}\sigma_q+c_2\right)\,...\,h_n\,\left(x^{\kappa_n}\sigma_q+c_n\right)\,.
\end{equation}
Furthermore, let $K^+=\ZZ_{>0}\cap\{\kappa_j:1\le j\le n\}$, and define $\vec k$  as follows:
\begin{equation}
 \vec k=\left\{ 
 \begin{array}{ll}\emptyset
  &\textrm{if}\quad K^+=\emptyset\\
  (k_1,...,k_m)& \textrm{if}\quad K^+=\{k_\ell: 1\le \ell\le m\} \ \textrm{and}\ k_1<...<k_m\,.
 \end{array}
\right. 
\end{equation}
By considering \cite[\S~3.3.5]{MZ}, it follows that any power series $\hat f\in\CC[[x]]$ such that $L(\hat f)\in\CC\{x\}$ remains convergent or is $Gq$-summable of order $\vec k=(k_1,...k_m)\in\ZZ_{>0}^m$ depending on whether $\vec k=\emptyset$ or not.

\begin{theorem}\label{thm:summability}
 Given any integer $n\ge 2$, the $n$-th power $\hat E_q(x)^n$ is $Gq$-summable of order $(1,2,...,n)$.
 And this is the same for $\hat E(x;q)^n$.
\end{theorem}

 \begin{proof}
  This follows from taking into account  \eqref{equation:Eulern} or applying Theorem \ref{thm:gen} to \eqref{equation:Eulerq} with the help of Remark \ref{rem:ab}.
 \end{proof}

 Applying \cite[Theorem 36]{Ba} or \cite[Proposition 1.3.4.2]{Ma} implies that  every power $\hat E(x)^n$ of the Euler series is Borel-summable at the same level, contrarily to what happens in the case of their respective $q$-analog $\hat E(x;q)^n$; see Theorem \ref{thm:summability} in the above. Furthermore, thanks to \cite[Theorem 3.15]{DVZ}, one knows that the $Gq$-sum of $\hat E(x;q)$ tends toward the Borel-sum of $\hat E(x)$ when $q\to 1^+$. In our coming work \cite{DZ}, it would be shown that, given one generic direction $d$ and two power series $\hat f_1$ and $\hat f_2$ such that $L_j(\hat f_j)\in\CC\{x\}$ for some $L_j\in\CC\{x\}[\sigma_q]$, where $j=1$ or $2$, the $Gq$-sum of $\hat f_1\hat f_2$ along $d$ is equal to the product of the $Gq$-sum of $\hat f_1$ along $d$ with that of $\hat f_2$. In this way, one would obtain that the $Gq$-sum of $\hat E_q(x)^n$ or $\hat E(x;q)^n$ can be really expressed as the $n$-th power of that of $\hat E_q(x)$ or $\hat E(x;q)$, respectively.

\end{document}